\documentclass[11pt]{amsart}
\advance\hoffset by -0.5 truecm
\usepackage{amssymb}
\newtheorem{Theorem}{Theorem}[section]
\newtheorem{Lemma}[Theorem]{Lemma}
\newtheorem{Corollary}[Theorem]{Corollary}

\newtheorem{Proposition}[Theorem]{Proposition}

\newtheorem{Definition}[Theorem]{Definition}
\newtheorem{Example}[Theorem]{Example}

\def \dim{{\mbox {dim}}\,}

\def\V{\mbox{Var}}

\def\R\re
\def\V{\bf V}

\def \re{{\mathbb R}}

\def \0{\lambda_{0}}

\begin{document}
\title[Isoparametric hypersurfaces]{Isoparametric hypersurfaces  and metrics of constant scalar curvature}

\author[G. Henry]{Guillermo Henry}\thanks{G. Henry is supported
by a postdoctoral fellowship of CONICET}
  \address{Departamento de Matem\'atica, FCEyN, Universidad de Buenos 
Aires, Ciudad Universitaria, Pab. I., C1428EHA,
           Buenos Aires, Argentina.}
\email{ghenry@dm.uba.ar}

\author[J. Petean]{Jimmy Petean}\thanks{J. Petean is supported
by grant 106923-F of CONACYT}
 \address{CIMAT  \\
          A.P. 402, 36000 \\
          Guanajuato. Gto. \\
          M\'exico \\
           and Departamento de Matem\'{a}ticas, FCEyN \\
          Universidad de Buenos Aires, Argentina (on leave).}
\email{jimmy@cimat.mx}


\date{}

%\maketitle

\begin{abstract} We showed the existence of  non-radial solutions of the equation
$\Delta u -\lambda u + \lambda u^q =0$ on the round sphere $S^m$, for $q<(m+2)/(m-2)$, 
and study the number of
such solutions in terms of $\lambda$. We show that for any isoparametric hypersurface 
$M\subset S^m$ there are solutions such that $M$ is a regular
level set (and the number of such solutions increases with $\lambda$). 
We also show similar results for isoparametric hypersurfaces in general
Riemannian manifolds.
These solutions give multiplicity results
for metrics of constant scalar curvature on conformal classes of Riemannian products.

\end{abstract}

\maketitle

\section{Introduction}Given a  Riemannian metric $g$  on a closed manifold $M^n$ 
we denote by $[g]$ the family of metrics conformal to $g$. The
classical Yamabe problem consists of finding metrics of constant scalar curvature in $[g]$. 
To study this problem H. Yamabe considered \cite{Yamabe} what we will now call
 the Yamabe constant of $[g]$ and denote by $Y(M,[g])$, the infimum
of the (normalized) total scalar curvature functional restricted to $[g]$:

$$Y(M,[g])  = \inf_{h\in [g]} \ \frac{ \int_M  s_h \  dvol_h }{ Vol(M,h)^{\frac{n-2}{n}}}.$$

\noindent
Here $s_h$ and $dvol_h$ denote the scalar curvature and volume element of $h$.

The critical points of the total scalar curvature functional restricted to $[g]$ are the metrics in
$[g]$ which have constant scalar curvature. Yamabe attempted to prove the existence
constant scalar curvature metrics in $[g]$ by showing that $Y(M,[g])$ is realized.
His proof contained a mistake but his statement was eventually proved to be correct
in a series of beautiful articles by N. Trudinger \cite{Trudinger}, T. Aubin \cite{Aubin}
and R. Schoen \cite{Schoen}.  It is not difficult to see that in $[g]$ there cannot exist two metrics
of constant scalar curvatures of different signs. Moreover, if there is a metric $h\in [g]$ of
constant non-positive scalar curvature then any other metric in $[g]$ of constant scalar curvature
is of the form $k g$ for some $k\in {\re}_{>0}$. 
But if $Y(M,[g]) >0$ the situation is much more interesting. For instance S. Brendle showed in
 \cite{Brendle} that in high dimensional spheres
there are conformal classes of metrics (different form the conformal class of the round metric)
for which the space of unit volume constant
scalar curvature metrics is a non-compact family.  
Finding metrics of constant scalar curvature in a conformal class $[g]$ amounts
to solving what is called the {\it Yamabe equation} for $(M,g)$: if we let $p= p_n =2n/(n-2)$ then $h=u^{p-2} g$
has constant scalar curvature $\lambda$ if and only if 

$$-a_n \Delta_g u + s_g u = \lambda u^{p-1} ,$$

\noindent
where $a_n =4(n-1)/(n-2)$. 
Expressing any metric in $h\in [g]$ as $h=u^{p-2} g$ the total scalar curvature functional
restricted to $[g]$ becomes

$$u\mapsto  \frac{\int_M a_n {\| \nabla u \| }^2  +s_g u^2  \ dvol_g }{ {\| f \|}_p^2 }$$

\noindent
and the Yamabe equation is the Euler-Lagrange equation of this functional.  To find all
solutions of the Yamabe equation in a fixed conformal class of positive Yamabe constant
is an extremely difficult problem.
One does have uniqueness of solutions in the case
of the conformal class of an Einstein metric by a classical result of M. Obata \cite{Obata},
and this  covers a big family of interesting examples. 
The only non-trivial
case which is completely understood is the case of cylinders $(S^n \times S^1 ,[g_0^n + 
T^2 dt^2 ])$: here (and in the rest of this article) $g_0^n$ denotes the round metric of
sectional curvature 1 on the $n$-sphere, $dt^2$ is the metric of diameter 1 in the circle and 
$T\in \re_{>0}$. This case was studied by R. Schoen \cite{Schoen2} and O. Kobayashi
\cite{Kobayashi, Kobayashi2}: in this case all solutions are constant along the spheres
$S^n$ and the number of solutions increase with $T$.

The principal interest in this article is to study multiplicity results for solutions of
the Yamabe equation  on the  products $(S^n \times S^k , 
[g_0^n + T^2 g_0^k ])$, for $n,k\geq 2$.  This problem was already considered
by the second author in \cite{Petean}, where it is studied the number of solutions 
 which depend only on the first variable, and are radial
(i.e. invariant by the canonical $SO(n)$-action which leaves the poles fixed). The
same type of results had actually been obtained previously (and in a  more
complete way) by Q. Jin, Y. Y. Li and H. Xu
\cite[Theorem 1.3]{Jin} in another context.

In this article we will show  the existence
of non-radial solutions of the equation. Since $2<p_{n+k} <p_n$ it is enough to
study solutions of the subcritical Yamabe equation on $S^n$: 

 \begin{equation}\label{SY} 
 - \Delta_{g_0} u + \lambda  u = \lambda u^q
 \end{equation}

\noindent
where $\lambda $ is a positive constant (related to the scalar curvature of the product manifold) 
and $1<q<p_n -1$. This equation has already been studied by several authors.
Of particular interest to us are the articles
\cite{Jin} that we just mentioned and the article \cite{Veron}, where M-F.Bidaut-Veron and
L. Veron show that the constant $u\equiv 1$ is the only solution of the equation if
$\lambda \leq n/(q-1)$. We will actually make use of this result later.

\vspace{.2cm}

To discuss our main results let us recall that
if $(M,g)$ is a  Riemannian manifold a smooth function $f:M \rightarrow \re$ is called an 
{\it isoparametric function} if there exist  a smooth function $a$ and a continuous function $b$
such that

$$\| \nabla (f) \|^2 = b \circ f $$

$$ \Delta_{g} (f) = a\circ f .$$  

Regular level sets of isoparametric fuctions are called {\it isoparametric hypersurfaces}.
The study of isoparametric hypersurfaces has a long history.  E. Cartan
\cite{Cartan} proved that in 
the case of space forms a hypersurface $M$ is isoparametric (according to the previous 
definition) if and only if it has constant principal curvatures.
For instance, isoparametric hypersurfaces in Euclidean space $\re^n$
are just the canonical embeddings of $S^{n-1}$, $\re^{n-1}$ or $S^k \times \re^{n-k-1}$
as was shown by T. Levi-Civita \cite{Levi-Civita} (in the three-dimensional case)
and B. Segre \cite{Segre}. But the situation is much more interesting 
in the case of the round sphere. The simplest examples are the regular orbits of codimension one
isometric actions on the sphere. These are called {\it homogeneous} and include many  
interesting examples;  a classification of homogeneous isoparametric hypersurfaces in
the sphere was given by
W. Y. Hsiang and H. B- Lawson \cite{Lawson}. 
But there are families of examples of non-homogeneous isoparametric hypersurfaces.
The first examples were found by H. Ozeki and M. Takeuchi \cite{Ozeki-Takeuchi-I, Ozeki-Takeuchi-II}.
Although many results have been obtained towards a classification of isoparametric
hypersurfaces in the spheres (see for instance the articles by T. Cecil, Q. S. Chi and G. Jensen \cite{Cecil},
S. Immervoll \cite{Immervoll} and R. Miyaoka \cite{Miyaoka})
a complete classification is still missing. 

We will show that given any isoparametric hypersurface $S \subset S^n$ and any $k\geq 2$,
for any positive $T$ such that 

$$\frac{1}{T} > \frac{6(n+5)(n+k-1) - n(n-1)}{k(k-1)} ,$$ 

\noindent
there exists a solution $u$ of the Yamabe equation on
$(S^n \times S^k , g_0^n + T g_0^k )$ so that $S \times S^k$ is a level surface
of $u$. Therefore to understand  all solutions to the Yamabe equation in Riemannian  products
of spheres  one needs to classify  all isoparametric hypersurfaces of the sphere; and the
previous comments then give us an idea of the complexity of the problem.

To state our results more precisely, given an isoparametric hypersurface
$S \subset S^n$ we let $l$ be number of distinct principal curvatures of $S$.
Then $l$ can only be 1, 2, 3, 4
or 6 as follows from the work of H. F. M\"unzner \cite{Munzner, Munzner2}. Moreover,
there is an isoparametric function $f$ associated to $S$ (i.e. $S$ is a regular level
set of $f$) which is the restriction to $S^n$ of a homogeneous polynomial $F$ in
$\re^{n+1}$ of degree $l$.
In this situation we will say that $S$ has degree $l$. 
Let  $\lambda_{i,n}  = -i(n+i-1)$ (these are the eigenvalues of the Laplace operator
 on $S^n$) and 
 
 $$\lambda_{i,n}^{l,q} = \frac{-\lambda_{il,n} }{q-1}.$$

 Note that for $l, q, n$ fixed the sequence is increasing in $i$.
 We will prove

\begin{Theorem}  Let $S$ be an isoparametric hypersurface of $S^n$ of degree
$l$. Then for each positive integer $i$ there exist at least $i$  solutions of 
equation (1) on $S^n$ for $\lambda \in (\lambda^{l,q}_{i,n}  \ , \ \lambda^{l,q}_{i+1,n} ]$
which are constant along $S$.
\end{Theorem}

In terms of solutions of the Yamabe equation in $(S^n \times S^k , g_0^n + T g_0^k )$,
$\lambda$ corresponds to $(1/a_{n+k}) (n(n-1)+(1/T) k(k-1))$. To be more explicit 
about what the theorem is saying in terms of multiplicity results for metrics of
constant scalar curvature let us first consider the case $n=k=3$. In $S^3$ there are
two types of isoparametric hypersurfaces, corresponding to the canonical 
isometric actions of $SO(3)$ and of $SO(2)\times SO(2)$: the first ones have $l=1$
and for the second ones $l=2$. Note that $a_6 = 5$ and $p_6 =3$. Let $\mu_i = -i(i+2)$
be the eigenvalues of the Laplace operator on $S^3$  and let 
$T_i = 6/(-5 \mu_i -6 )$ for $i\geq 1$, $T_0 =1$. Then Theorem 1.1 translates to:

\begin{Corollary}  For $T\in [ T_{i+1}, T_i )$ there are $i +[i/2]$ non-isometric unit volume 
metrics of constant scalar curvature in $[g_0^3 +T g_0^3 ]$.

\end{Corollary}

The generic connected components of the level surfaces of the corresponding solutions
of the Yamabe equation are diffeomorphic to either $S^2 \times S^3$ or
$S^1 \times S^1 \times S^3$.

\vspace{.2cm}

Let us now consider the general case $n,k\geq 2$. One needs to count the number of known isoparametric hypersurfaces
(to be more precise one should say families of isoparametric hypersurfaces: each isoparametric hypersurfaces has associated
an isoparametric function $f$ and then the family of isoparametric hypersurfaces given as the regular level sets of $f$).
For a positive integer  $m$  let $j_0$ be the number of 2's that appear in the factorization of $m+1$. If  $j_0=4l+d$ with $d=0,1,2,3$,
 define $\gamma(m)= 8l+2^d$. Also, for a positive integer $r$ 
 let $\phi(r)$ be the number of integers $s$ such that $1\leq s\leq r-1$ and $s\equiv0,1,2,4$  $mod(8)$. And define 
$\beta(m)=\sum_{   \{ m_1 \  : \  \phi(m_1) \leq j_0  \ , \  m_1 \equiv 0 (4)  \}  }        [\frac{m+1}{2^{\phi (m_1 )+1}}].$

\vspace{.1cm}

Let $\mu_{i}^{n,k} = (n+k-1) \ \lambda_{i,n}$
and  $T_{i}^{n,k} = k(k-1)/(-\mu_{i}^{n,k}-n(n-1) )$, for $i\geq 1$. We have:

\begin{Corollary}
Let $N_Y^{n,k} (T)$ be the number of isometrically distinct unit volume metrics of constant scalar curvature in
the conformal class of $(S^n \times S^k , g_0^n + T g_0^k )$ and suppose that $T\in [T_{i+1}^{n,k},T_{i}^{n,k})$. 
Then:
\begin{itemize}
\item [a)] If $n=2m$ with $m\neq 2$ then $N_Y^{n,k} (T) \geq i +[(2m-1)/2][i/2]$.  
\item [b)] If $n=2m+1$ with $m\neq 1, 3, 4,  6, 7, 9, 12$ then $N_Y^{n,k} (T) \geq i+m[i/2]+(\gamma(m) + \beta (m) ) [i/4]$. 
\item [c)] \begin{enumerate}
\item If $n=3$  then  $N_Y^{n,k} (T) \geq i+[i/2]$.
\item If $n=4$ then  $N_Y^{n,k} (T) \geq i+[i/2]+[i/3]$.
\item If $n=7$ then  $N_Y^{n,k} (T) \geq i+3[i/2]+[i/3]+[i/4]+[i/6]$.
\item If $n=9$ then  $N_Y^{n,k} (T) \geq i+4[i/2]+2[i/4]$.
\item If $n=13$ then  $N_Y^{n,k} (T) \geq i+6[i/2]+[i/3]+[i/4]+[i/6]$.
\item If $n=15$ then  $N_Y^{n,k} (T) \geq i+7[i/2]+4[i/4]$.
\item If $n=19$ then  $N_Y^{n,k} (T) \geq i+9[i/2]+3[i/4]$.
\item If $n=25$ then $N_Y^{n,k} (T) \geq i+ 12[i/2]+[i/3]+[i/4]$.
\end{enumerate}

\end{itemize}
\end{Corollary}

Isoparametric hypersurfaces of degree 1  in $S^n$ are homogeneous, the orbits of the action of $O(n)$, the corresponding solutions
to the Yamabe equation are the radial solutions discussed in \cite{Jin, Petean}. These solutions account for all the $i$'s in the corollary.
Similarly isoparametric hypersurfaces of degree 2 are  $S^{k}\times S^{n-k-1}$, the orbits of the isometric actions of
$O(k+1) \times O(n-k)$, with $1\leq k\leq [n-1/2]$: these solutions account for all the $[i/2]$'s in the corollary. 
Isoparametric hypersurfaces of  degree 3 were already classified by E. Cartan \cite{Cartan}: in this case the three principal curvatures
have the same multiplicity which  can be  1, 2, 4 or 8. The corresponding isoparametric hypersurfaces in $S^4$,
$S^7$, $S^{13}$ and $S^{25}$ are called
Cartan hypersurfaces and they are tubes over the canonical embeddings of the projective planes ${\bf FP}^2$ (where
${\bf F}$ are the real, complex, quaternionic or Cayley numbers). These account for all the $[i/3]$'s in the corollary.
There are essentially two  isoparametric hypersurfaces of degree 6 in the sphere, one in $S^{7}$ and the other in $S^{13}$. 
In both cases all the principal curvature have the same multiplicities,  1 and 2 respectively 
(see U. Abresch \cite{abresch} and R. Miyaoka \cite{Miyaoka}). 
Isoparametric hypersurfaces of degree 4 provide the richest examples.  The $[i/4]$'s appearing in the corollary come from a careful
counting of all known isoparametric hypersurfaces of degree 4. We give the details on this counting at the end of next 
section. It is known that  new examples of  isoparametric hypersurfaces of degree 4 can only appear in $S^{15}$, $S^{19}$ y $S^{31}$,
 see T. Cecil, Q. S.  Chi and G. Jensen \cite{Cecil}. Therefore in all other dimensions we are counting all solutions of the Yamabe
 equation provided by Theorem 1.1.

\vspace{.5cm}

For a general closed Riemannian manifold $(M,g)$ with an isoparametric hypersurface
$S\subset M$ we will show in Section 3 that there is a subsequence of the sequence of eigenvalues
of the Laplace operator on $(M,g)$ such that for each of the eigenvalues in the
subsequence there exists an associated  non-zero eigenfunction which is constant along  $S$. 
Let us call $\lambda_0 = \lambda_0 (M,g,S)<0$ the first such eigenvalue.

\begin{Theorem}  Assume that $(M^n ,g)$ is a closed Riemannian $n$-manifold of constant 
scalar curvature and $S$ an isoparametric hypersurface as above.
If $(N^k ,h)$ is a Riemannian $k$-manifold of constant scalar curvature let
${\bf s} = s_g + s_h$ be the scalar curvature of the product.
If $ {\bf s} >\frac{-a_{n+k} \lambda_0 }{p_{n+k}-2}$  there is a smooth function 
$u: M \rightarrow \re$ which is constant along $S$ and  solves the Yamabe 
equation for $(M\times N, g+h)$.
\end{Theorem}

\section{Some observations  on isoparametric functions }

Let $(M,g)$ be a closed connected Riemannian manifold and $f:M \rightarrow [a,b]$ a smooth surjective
function. 
$f$  is called isoparametric if there exist  a 
smooth function $b: \re \rightarrow \re $ and  a continuous function $a: \re \rightarrow \re $ such that 
\begin{equation}\label{transnormal} \|\bigtriangledown f \|^2=b(f) \end{equation}        and \begin{equation}\label{isopara}\Delta f= a(f) \end{equation}

 Isoparametric functions have been studied for a long time, in particular by E. Cartan \cite{Cartan} in space
forms. We will mostly be interested in the case of the sphere but we will also use some  results 
about isoparametric functions
in general Riemannian manifolds: a good introduction to this general case is  the work of Q. M. Wang
\cite{Wang}.  

If $f:M \rightarrow [a,b]$ is an isoparametric function then $a$ and $b$ are the only critical
values of $f$. The level sets $f^{-1} (a) $ and $f^{-1} (b)$ are called the focal hypersurfaces
of $f$. For $t\in (a,b)$, the regular level sets of $f$, $f^{-1} (t)$ are called isoparametric hypersurfaces.

We are going to be interested in the family of functions which are constant along the level surfaces
of a fixed isoparametric function. We will use  the following notation:

\begin{Definition}
We call $S_f $ the family of 
functions on $S^n$ which are constant along level surfaces of $f$. 
\end{Definition}

{\bf Remark}: We have not said if the functions on $S_f$ are meant to be continuous, or
smooth, for instance. This is because we will need to use different function spaces. We
will be explicit about the regularity assumed for functions in $S_f$ when needed. 

\vspace{.2cm}

The most familiar case is to consider an isometric 
 codimension one action by a Lie group
$G$ on $(M,g)$. Then any smooth function $f$ for which the level sets are 
the orbits of the action is isoparametric
 (equivalently,  $f$ is the composition of the projection to the orbit space with an
 injective smooth function on the orbit space). Then $S_f$ is the family of
 $G$-invariant functions (or equivalently the family of functions obtained as the
 composition of $f$ with a function on the range of $f$).
 
 \vspace{.2cm}

In the reminder of this section we will discuss the case of the round sphere.
Therefore from now on let $f:S^n \rightarrow [a,b]$ be an isoparametric function on the sphere. 
Let $t$ be a regular value of $f$ and denote by
$S_t =  f^{-1} (t)$ the corresponding isoparametric hypersurface. It was shown by 
Cartan that  $S_t$ has
constant principal curvatures and this condition  characterizes isoparametric 
hypersurfaces in space forms; but we will not make use of this fact.

It is clear that if $f$ is an isoparametric function and $u:\re \rightarrow \re$ is a smooth strictly
monotone function then $u\circ f$ is also an isoparametric function. In the case of the sphere one
has a certain normalization: if $M \subset S^n$ is an isoparametric hypersurface then 
M\"{u}nzner proved in \cite{Munzner, Munzner2} that there exists a homogeneous
polynomial $F: {\re^{n+1}} \rightarrow \re$ of degree $l$ which satisfies what are known as the
{\it Cartan-M\"unzner equations}:

 \begin{equation}\label{munznerI} <\nabla F,\nabla F>=l^2\|x\|^{2l-2}
\end{equation}
\begin{equation}\label{munznerII}
\Delta F=\frac{1}{2}cl^2\|x\|^{l-2} ,
\end{equation}
where $c$ is an integer which is given below,
such that  $M$ is  a regular  level set of $f={F _| }_{S^n}$, which is an isoparametric function. 
Moreover, $l$ can only take the values 1, 2, 3, 4 or 6 and
coincides with the number of distinct principal curvatures of $M$. In case $l = 3$ (or 1) the principal curvatures have the same
multiplicities and in case $l=2, 4$ or $6$ there are integers $m_1$ and $m_2$ (which might be equal) such that
half of the principal curvatures have multiplicity $m_1$ and the other half $m_2$; in particular
$(l/2) (m_1 +m_2 )= n-1$. 
Then the constant $c$ is
$m_2 - m_1$. The polynomial $F$ is called a {\it Cartan-M\"unzner polynomial}.
Note that interchanging $m_1$ and $m_2$ corresponds to replacing $F$ 
with $-F$. In this situation we will say that $f$ is an isoparametric function
of {\it degree $l$} and similarly that $M$ is an isoparametric hypersurface of
{\it degree $l$}.

Recalling that the Laplace
operator $ \Delta_{S^n}$ on the sphere relates to the Laplacian on Euclidean space by the formula

$$\Delta_{S^n} (u) =  \left( \Delta U -\frac{\partial^2 U}{\partial r^2} \right) -\ n \ \frac{\partial U}{\partial r} $$ 

\noindent
and

$$ \| \nabla (U) \|^2  = {\left( \frac{\partial U}{\partial r} \right)}^2 + \| \nabla (u) \|^2 $$

\noindent
(where $r= \| x\| $, $U:\re^{n+1} \rightarrow \re$ and $u=U|_{S^n}$) it is easy to check that 
$$\| \nabla (f) \|^2 = b \circ f $$

$$ \Delta_{S^n} (f) = a\circ f ,$$  

\noindent
$a(t)= -l(n+l-1)t+(1/2) c l^2$, $b(t)=-l^2 t^2  + l^2$. Note that since the only zeros of $b$ are
$\pm 1$ it follows that the range of $f$ is $ [-1,1]$. 

\vspace{.5cm}

Let us see a few examples of isoparametric functions in spheres and the corresponding 
solution of the Cartan-M\"unzner equations.

\begin{Example}\label{x_n+1} The simplest example of f a 
Cartan-M\"unzner polynomial,   is given  by $f(x)=f(x_1,\dots,x_{n+1})=x_{n+1}$. Then
$\|\bigtriangledown f \|=\sqrt{1-x_{n+1}^2}=\sqrt{1-f^2}$ and $\Delta f=nx_{n+1}=nf$. $f$ is invariant
under the canonical $SO(n)$ action that fixes the poles and $S_f$ is just the family of radial functions.
\end{Example}

\begin{Example} Consider the obvious $O(n) \times O(k)$-action on $S^{n+k-1} \subset 
{\bf R^{n+k}}$. For $(x,y)\in {\bf R^{n+k}}$ write $x^2 =x_1^2 +...+x_n^2$ and
$y^2 =y_1^2 +...+y_k^2$. Then $x^2$ and $y^2$ are homogeneous polynomials of
degree 2 invariant under the action. Let $F= x^2 -y^2$. Then $F$ is a Cartan-M\"unzner
polynomial with $l=2$, $m_1 = k-1$, $m_2 = n-1$. The corresponding isoparametric hypersurfaces 
have two principal curvatures and are diffeomorphic to $S^{n-1} \times S^{k-1}$. As before we
denote by $f=F_{|_{S^n}}$; then $S_f$ is the family of $O(n) \times O(k)$-invariant functions.
\end{Example}

\begin{Example}\label{exmpl2} If $z\in \re^{2n+2}$ we write  $z=(x,y)\in \re^{n+1}\times \re^{n+1}$. 
Consider the homogeneous polynomial  $F:\re^{2n+2} \rightarrow \re$  defined by $F(z)=\|  z\|^4 -2(\|x\|^2-\|y\|^2)^2 -8(<x,y>)^2$.
Then $F$ is a Cartan-M\"unzner
polynomial with $l=4$, $m_1 = 1$, $m_2 = n-1$. It is proved by K. Nomizu in \cite{Nomizu}
that the four distinct principal curvatures are $\frac{1+\sin(2t)}{\cos(2t)}$,   $\frac{-1+\sin(2t)}{\cos(2t)}$, $\tan 
(t)$ y $-\cot (t)$  (the first two have multiplicity $n-1$ and the second two have multiplicity 1).
 
 \end{Example}

An isoparametric hypersurface $M$ of $S^n$ is homogeneous if there is a suitable subgroup of $O(n)$  that acts on $M$ transitively.
These are the examples which are easier to understand. The previous examples are all homogeneous. And there exist much more
examples, of course. All homogeneous isoparametric hypersurfaces have actually been classified by W. H. Hsiang and H. B. Lawson in \cite{Lawson}.
But not all the isoparametric surfaces are homogeneous. The first examples were constructed by H. Ozeki and M. Takeuchi in
\cite{Ozeki-Takeuchi-I}:

\begin{Example}Let  ${\bf{H}}$ be the real quaternion algebra and let $u\longrightarrow \bar{u}$ be the natural involution. 
We can identify $\re^{16}$ with ${\bf{H}}^2\times\bf{H}^2$, we note $x\in \re^{16}$ as  $x=(u_0,u_1,v_0,v_1)$ 
where $u_i,v_i\in {\bf{H}}$ for $i=\{0,1\}$. Let $F:\re^{16}\longrightarrow \re$ defined by 
\begin{equation}F=\|x\|^4-2F_0(x)\end{equation}
 where

\begin{eqnarray}\label{nohomogeneus}
F_0(x)&=&4\Big(u_0\bar{v}_0+u_1\bar{v}_1\Big)\Big(\bar{u}_0v_0+\bar{u}_1v_1\Big) \nonumber \\
&-&\Big(u_0\bar{v}_0+u_1\bar{v}_1+v_0\bar{u}_0+v_1\bar{u}_1\Big)^2\\
&+& \Big[ u_1\bar{u}_1-  v_1\bar{v}_1+u_0\bar{v}_0+v_0\bar{u}_0\Big]^2 \nonumber
\end{eqnarray}

 In \cite{Ozeki-Takeuchi-I} and \cite{Ozeki-Takeuchi-II}, Ozeki and Takeuchi showed that the function $F$ 
 (that satisfied the equations (\ref{munznerI}) and (\ref{munznerII}) with $l=4$ and $c=1$) produced non homogeneous isoparametric hypersurfaces in $S^{15}$.

\end{Example}

We now study the number of distinct isoparametric hypersurfaces of degree 4 which are known. This is the counting we used for Corollary 1.3.

In the article \cite{Ferus} Ferus, Karcher and M\"unzner  generalized the construction of Ozeki and Takeuchi to produce examples of homogeneous
and non-homogeneous isoparametric hypersurfaces of degree 4: these are known as hypersurfaces of Ferus-Karcher-M\"unzner type
or F-K-M hypersurfaces. The only known examples of isoparametric hypersurfaces of degree 4 which are not F-K-M are two
homogeneous examples with multiplicities $(2,2)$ and $(4,5)$. S. Stolz \cite{Stolz} proved that the multiplicity of any 
isoparametric hypersurface of degree 4 is equal to one of these known examples. 
The multiplicities of F-K-M   hypersurfaces are of the form $(m_1,m_2 )=(m_1 ,l\delta(m_1)-m_1-1)$ where $\delta(m_1)$  is the unique positive integer such that the Clifford algebra of $m_1$ elements $C_{m_1-1}$ has an irreducible representation on $\re^{\delta(m_1)}$ and $l$ is an integer such that  $m_2>0$, see T. Cecil, Q. S. Chi, G. R. Jensen \cite{Cecil}. It is known (see for instance  \cite{Stolz}) that $\delta(r)=2^{\phi(r)}$ where $\phi(r)$ is the number of integers $s$ such that $1\leq s\leq r-1$ and $s\equiv0,1,2,4$  $mod(8)$. 
$\phi(r)$ is of course cyclic $mod(8)$. As a simplification for the reader,  we present a table with the value of $\phi(r)$ for $1\leq r\leq 16$:

\begin{center}
\begin{tabular}{||cc||cc||}\hline
$r$&$\phi(r)$&$r$&$\phi(r)$\\ \hline
1&0&9&4\\
2&1&10&5\\
3&2&11&6\\
4&2&12&6\\
5&3&13&7\\
6&3&14&7\\
7&3&15&7\\
8&3&16&7\\\hline
\end{tabular}
\end{center}

The dimension of the hypersurface with multiplicities $(m_1,m_2)$ is  of course $2(m_1 + m_2 )$. So there are no isoparametric
hypersurfaces in even-dimensional spheres. We want to count the number of such hypersurfaces in $S^{2m+1}$. The homogeneous
hypersurfaces with multiplicities $(2,2)$ and $(4,5)$ provide one example in $S^9$ and one in $S^{19}$. The case $m_1 =1$ in the
previous discussion provides one F-K-M hypersurface for any $m\geq 2$.  Ferus, Karcher and M\"unzner proved in \cite{Ferus} that if
 $m_1\equiv 0$ $mod(4)$,  there are $[l/2]+1$ incongruent families of isoparametric hypersurfaces 
 (two families are called {\it congruent} if  the corresponding Cartan-M\"unzner polynomials are related by an isometry of the sphere)
 with  multiplicity
 $(m_1 , l\delta(m_1)-m_1-1)$. For all other $m_1$'s there is only one F-K-M.

We provide the table with  the pairs $(m_1,m_2)$ that  correspond to families of isoparametric surfaces of F-K-M type with multiplicities $(m_1,m_2,m_1,m_2)$
for small values of $m_1$ and $m_2$ to simplify checking Corollary 1.3 (c):

\begin{center}
\begin{tabular}{|c|l|}\hline
$(1,k-2)$&${(1,1)}, {\bf (1,2)}, (1,3), (1,4), {\bf(1,5)}, (1,6), (1,7), \dots$\\
$(2,2k-3)$&${\bf(2,1)}, (2,3), {\bf(2,5)}, (2,7), {(2,9)}, (2,11), (2,13), \dots$\\
$(3,4k-4)$&${\bf(3,4)}, (3,8), {(3,12)}, (3,16), {(3,20)}, (3,24), \dots$\\
$(4,4k-4)$&${\bf(4,3)}, (4,7), {(4,11)}, (4,15), {(4,19)}, (4,23),  \dots$\\
$(5,8k-6)$&${\bf(5,2)}, (5,10), {(5,18)}, (5,26), {(5,34)}, (5,42),  \dots$\\
$(6,8k-7)$&${\bf(6,1)}, {\bf(6,9)}, {(6,17)}, (6,25), {(6,33)}, (6,41),  \dots$\\
$(7,8k-8)$&${\bf(7,8)}, (7,16), {(7,24)}, (7,32), {(7,40)}, (7,48),  \dots$\\
$(8,8k-9)$&${\bf(8,7)}, (8,15), {(8,23)}, (8,31), {(8,39)}, (8,47),  \dots$\\
$(9,16k-10)$&${\bf(9,6)}, (9,22), {(9,38)}, (9,54), {(9,70)}, (9,86), \dots$\\
$(10,32k-11)$&${(10,21)}, (10,53), {(10,85)}, (10,107), {(10,149)}, \dots$\\
$(11,64k-12)$&${(11,52)}, (11,116), {(11,180)}, (11,244), {(11,308)},  \dots$\\
$(12,64k-13)$&${(12,51)}, (12,115), {(12,179)}, (12,243), {(12,307)},  \dots$\\
$(13,128k-14)$&${(13,114)}, (13,242), {(13,370)}, (13,498), {(13,626)},  \dots$\\
$\vdots$ & $\cdots$ \\\hline
\end{tabular}
\end{center}

The family that correspond to $(2,1)$, $(6,1)$, $(5,2)$ and one of the family of $(4,3)$ 
(from the comments above we know that there are two of these) 
are congruent 
to those with multiplicities $(1,2)$, $(1,6)$, $(2,5)$ and $(3,4)$ respectively, and these are the only congruences among the hypersurfaces  of F-K-M type, see T. Cecil \cite{Cecil2} and Ferus, Karcher, M\"unzner \cite{Ferus}. 
So for instance for $2m+1 =7$ the pair $(1,2)$ provides the only F-K-M hypersurface in $S^7$. The $(1,3)$ F-K-M and the homogeneous $(2,2)$ give the
two examples in $S^9$. 
The $(1,4)$ and $(2,3)$ F-K-M hypersurfaces provide the two examples in $S^{11}$.
The $(1,5)$ F-K-M hypersurface is the only example in $S^{13}$. In $S^{15}$ we have the F-K-M hypersurfaces of multiplicities
$(1,6)$, $(2,5)$, $(3,4)$ and (the extra) $(4,3)$. In $S^{17}$ we only have the $(1,7)$ F-K-M hypersurface and in $S^{19}$ we have the homogeneous
example with multiplicities $(4,5)$ and the F-K-M hypersurfaces of multiplicities $(1,8)$ and $(2,7)$.

Now assume that $m\geq 10$. If $m+1 =  2^{j} \ c$   and $m_1$ is such that  $\delta(m_1)=2^{j}$ then
we have an F-K-M hypersurface with multiplicities $(m_1 , 2^j c -m_1 -1)$ in $S^{2m+1}$, as long
as $m_1 <\delta(m_1) \ c -1$ (but this condition is certainly verified if $m_1>8 $ or $c>1$ and $m_1 \neq 1$. For $c=1$
we would have  $\delta(m_1) = m+1 \leq m_1 +1 \leq 9$. Therefore we can forget about this condition
when $m\geq 10$). 
Then we have one F-K-M hypersurface for any $m_1$ such that $\phi(m_1) \leq j_0$.
 Therefore if we write $j_0=4l+d$ with $l\geq 0$ and $d=0, 1, 2, 3$ then there are  $\gamma(m)=8l+2^d$ F-K-M
 hypersurfaces  in $S^{2m+1}$.

Finally if $\phi(m_1) \leq j_0$ and  $m_1\equiv 0$ $mod(4)$, we have $[l/2]+1$ incongruent F-K-M hypersurfaces with
$m_2=l \delta(m_1)-m_1-1$ (by the result in \cite{Ferus} mentioned above). We already counted one of them. 
Note that $l=(m+1)/2^{\phi(m_1)}$. Then we call 
$\beta(m)=\sum_{   \{ m_1 \  : \  \phi(m_1) \leq j_0  \ , \  m_1 \equiv 0 (4)  \}  }        [\frac{m+1}{2^{\phi (m_1 )+1}}].$

It follows that we  have  at least $\gamma(m)+\beta(m)$ different families of isoparametric hypersurface of degree 4 in $S^{2m+1}$ with $m\geq 10$.

\section{Eigenvalues of the Laplacian}

Let $(M,g)$ be a closed Riemannian manifold and $f:M \rightarrow \re$
an isoparametric function. Then there exist eigenfunctions of the Laplace operator
of $(M,g)$ which belong to $S_f$. Let us first observe the following:

\begin{Lemma} If $u\in S_f$ is smooth then  $\Delta_{g} u \in S_f$.
\end{Lemma}
\begin{proof} Let $u=\varphi \circ f$, for a smooth function $\varphi$. Then by a direct computation $\Delta_{g} u
= \varphi ''   b(f)+\varphi ' a(f) \in S_f$. 
\end{proof}

\begin{Proposition} There exist an infinite sequence   $0<\lambda_1<\lambda_2<\lambda_3<\dots <\infty$    such that 
there exists an eigenfunction $f_i \in S_f$ of $\Delta_{g}$ with eigenvalue
$-\lambda_i$ .
\end{Proposition}
\begin{proof} The constant function $1\in S_f$  is an eigenfunction of the zero eigenvalue. 
We will use the usual arguments  with Rayleigh quotients  to characterize the following 
eigenfunctions with non zero eigenvalues in $S_f$. 
Let \begin{equation}A_1=\Big\{h\in S_f \cap  H^2_1 (M ):\ \int h \ dvol_{g}=0 \Big\},\end{equation} i.e.  the functions in  $S_f \cap  H^2_1 (M)$
orthogonal to the first eigenfunction ($H^2_1$ means of course  the Sobolev space with norm $\|u\|_1^2=\|u\|_2+\|\nabla u\|_2$).
 Consider \begin{equation}\label{1steigenvalue}\lambda_1=\inf_{h\in A_1 - \{ 0 \}}\frac{\int_{M} \|\nabla h\|^2dvol_{g}}{\int_{M}h^2dvol_{g}} .
 \end{equation}
By the usual argument using  the  Rellich-Kondrakov  theorem  we can find a minimizing sequence  $h_i \in A_1 - \{ 0 \} $ that converges to 
a minimizer $f_1 \in H^2_1 (M)$ with  $\| f_1  \|_2=1$,  $\int f_1 \  dvol_{g} =0 $ and $\|\nabla f_1  \|_2^2  =   \lambda_1$. Since $A_1$ is a closed
subspace of  $H^2_1 (M)$ it follows that $f_1 \in A_1 -\{ 0 \}$ is a minimizer and so a critical point of the functional in $A_1 -\{ 0 \}$.
This means that for all $u \in A_1$,

\begin{equation} \int_M \ (\Delta_g f_1 + \lambda_1 f_1 ) \ u  \ dvol_g =0
\end{equation} and since $\Delta_g f_1 + \lambda_1 f_1 \in A_1$ by Lemma 3.1, it follows that $f_1$ is 
eigenfunction of $\Delta_{g}$ with eigenvalue $-\lambda_1$. Then one defines $A_2 = A_1 \cap <f_1 >^{\perp}$ and
repeating the argument obtains and eigenfunction $f_2$ and so on. 

\end{proof}

The sequence $0>-\lambda_1 >-\lambda_2 >\dots $ is of course a subsequence of the spectrum of $\Delta_{g}$. One can also
see that:

\begin{Proposition} For each eigenvalue $\lambda_i$ the space of $S_f$-eigenfunctions  has  dimension one. 
\end{Proposition}

This is an easy consequence of the fact that $S_f$-eigenfunctions are obtained as solutions of a second order
ordinary differential equation. We will do it again in the case of the round sphere, which is our main interest, so we
omit the simple proof here.

\vspace{.2cm}

In the case of an isoparametric function on the round sphere one
can be much more precise. 
We look for eigenfunctions of the Laplace  operator
on $(S^n ,g_0 )$ which belong to $S_f$ for some isoparametric function
$f$. Recall that the eigenvalues of $\Delta_{S^n}$ are
$\lambda_{i,n} = -i(n+i-1)$. $f$ is an eigenfunction with eigenvalues
$\lambda_i$ if and only if $f$ is the restriction to the sphere of a
harmonic homogeneous polynomial in $\re^{n+1}$
of degree $i$.  

\begin{Lemma} Let $f:S^n \rightarrow [-1,1]$ be an isoparametric function
obtained as the restriction to $S^n$ of a solution $F$ of de Cartan-Munzner
equations. Let $l$ be the degree of $F$. Then for each $i=0,1,...$ there
is an eigenfunction $f_i \in S_f$ of the Laplacian $\Delta_{S^n}$ with eigenvalue
$\lambda_{il,n}$. The space of such eigenfunctions has dimension 1 and
is generated by $p_i \circ f$ where $p_i$ is a monic polynomial of degree
$i$ which has $i$ distinct simple roots in the interval $(-1,1)$.  Moreover,
if $\lambda_{j,n}$ is an eigenvalue of $\Delta_{S^n} |_{S_f}$ then $j=il$ for
some $i$.
\end{Lemma}

\begin{proof}  $F$ is a homogeneous polynomial of degree $l$ in
$\re^{n+1}$ which solves  the equations (\ref{munznerI}) and (\ref{munznerII}).

If $l$ is odd then $c= m_2 -m_1 =0$ and we let $U=F$. If $c \neq 0$
we let $U=F-(c/(n+1)) (x_1^2 +...+x_{n+1}^2 )^{l/2}$. Then $U$ is a harmonic homogeneous polynomial 
of degree $l$.
Let $u=U|_{S^n}$. Then $u\in S_f$ is an eigenfunction of $\Delta_{S^n}$ with eigenvalue
$\lambda_{l,n}$. And of course $u = p_1 (f)$ with $p_1 (t) = t-(c/(n+1))$. Then $u$ verifies
 $S_u = S_f$, $\Delta u=\lambda_{l,n} \  u$, ${\| \nabla u \|}^2 = {\| \nabla f \|}^2 = l^2 (1-f^2 )
 =l^2 (1-(u+c/(n+1))^2 )$.

Now a function $\alpha \circ f$ is eingenfunction of $\Delta_{S^n}$ with
eigenvalue $\lambda_{i,n}$ if and only if $\alpha$ solves the second order
ordinary differential equation

$$O_i (\alpha )= l^2 \alpha '' (t) (1-t^2) +  \alpha ' (t)  ((1/2) c l^2 -l(n+l-1) t)- \lambda_{i,n}  \alpha (t) =0 .$$

By a straightforward computation 

$$O_{il} (t^i ) = l^2 i(i-1) t^{i-2} (1-t^2 ) +i t^{i-1} ((1/2) c l^2 -l(n+l-1) t) +il(n+il-1) t^i = $$

$$= Ct^{i-1} +Dt^{i-2} $$

\noindent
for some $C,D \in \re$. Moreover for $j<i$ $O_{il} (t^j ) = Et^j + Ft^{j-1} + Gt^{j-2}$ 
for some $E, F, G \in \re$ and $E\neq 0$. We set $p_0 = 1$, $p_1 (t) = t-(c/(n+1))$
(as before) and it follows that there is exactly one
monic polynomial of degree $i$, $p_i$ which solves $O_{il} (p_i )=0$.
The fact that the roots of $p_i$ are simple is a simple consequence of the
fact that $p_i$ is a non-trivial solution of a second order ordinary differential 
equation. The fact that
it has $i$ roots in $(-1,1)$ is clear for $i=0$ and $i=1$. Assume that it is
true for some $i$. Let $-1<t_1 <...<t_i <1$ be the $i$ roots. If we call $t_0 =-1$
and $t_{i+1} =1$ then it is enough to prove that $p_{i+1}$ has at least
one root in each interval $(t_j ,t_{j+1} )$. 
This is a classical application of Sturm comparison theorem
\cite[page 229]{Ince}: we have

$$  p_i ''+   \frac{((1/2) c l^2 -l(n+l-1) t)}{l^2 (1-t^2)} p_i ' 
-\frac{ \lambda_{il,n}}{l^2 (1-t^2)} p_i =0 ,$$

$$  p_{i+1} ''+   \frac{((1/2) c l^2 -l(n+l-1) t)}{l^2 (1-t^2)} p_{i+1} ' 
-\frac{ \lambda_{(i+1)l,n}}{l^2 (1-t^2)} p_{i+1} =0 ,$$

\noindent
with $-\lambda_{(i+1)l,n} > -\lambda_{il,n}$. Also

$$\frac{p_{i+1} '(-1)}{p_{i+1} (-1)} = \frac{ \lambda_{(i+1)l}}{((1/2) c l^2 +l(n+l-1) )} <$$
$$<  \frac{ \lambda_{il}}{((1/2) c l^2 +l(n+l-1) )} =\frac{p_{i} '(-1)}{p_{i} (-1)}.$$

\noindent
Then if $p_{i+1}$ did not have a 0 between $-1 = t_0$ and $t_1$ we would apply Sturm's Theorem to reach
a contradiction. In a similar way one checks that $p_{i+1}$ must have a zero in each of the
other intervals $(t_j , t_{j+1} )$; and this completes the proof that $p_i$ has exactly $i$
zeros in the interval $(-1,1)$.

To prove the last statement in the lemma, pick $j$, $il<j<(i+1)l$. Let $\varphi $ be a non-trivial
solution of $O_{j} (\varphi )=0$. Then applying the same Sturm comparison to $p_i$ and
$\varphi$ would prove that $\varphi$ has at least $i+1$ zeros. And then applying the same
Sturm comparison to $\varphi$ and $p_{i+1}$ would prove that $p_{i+1}$ has at least 
$i+2$ zeros, which is of course false.

\end{proof}

\section{Proof of Theorem 1.1}

We will use bifurcation theory to prove Theorem 1.1. We will use the notation in
\cite{Nirenberg}. This is the same analysis carried out in \cite{Jin} to study
solutions which are radial with respect to some axis, but we will see that it
works in this more general case. The proof is based on Krasnoselski's Theorem
and the global extension by P. Rabinowitz; this corresponds with sections 3.3 and
3.4 of \cite{Nirenberg}. 

To begin with let us recall the following

\begin{Definition} 
Given a Banach space $X$  and a $C^r$ map $H:X\times \re \rightarrow
X$ a point $(0,\lambda_0  )$ such that $H(0,\lambda_0 )=0$ is called a
{\it bifurcation point} if every neighborhood of $(0,\lambda_0 )$ contains points
$(x,\lambda )$ with $x\neq 0$ such that $H(x,\lambda )=0$. 
\end{Definition}

Our first task is to write our equation as an operator equation as in 
bifurcation theory. Fix $q$, $1<q<p_n -1$. 
Given a positive solution $u:S^n \rightarrow \re_{>0}$ of  equation (1)

$$-\Delta u +\lambda u = \lambda u^q$$

\noindent
we let $v = u-1$; then $v>-1$ and it is a solution of $Eq_{\lambda}$:

$$-\Delta v = \lambda ((v+1)^q - (v+1)).$$

Now fix an isoparametric function $f:S^n \rightarrow [-1,1]$, where $f = F|_{S^n}$  for
 a homogeneous polynomial $F$ of degree $l$ satisfying the
Cartan-M\"unzner equations. And consider the
Banach space

$$C^{2,\alpha} (S_f ) = C^{2,\alpha } (S^n ) \cap S_f .$$ 

Let $T:C^{2,\alpha} (S_f ) \rightarrow C^{2,\alpha} (S_f )$ be the inverse
of $-\Delta + Id : C^{4,\alpha} (S_f )\rightarrow  C^{2,\alpha} (S_f )$; $T$ is a linear compact map. 
Consider the region

$$A=\{ (v,\mu )\in C^{2,\alpha} (S_f )  \times \re : \mu >1 , v>-1 \} ,$$

\noindent
relate  $\lambda$ and $\mu$ by the equation $\mu = (q-1) \lambda +1$ 
and define $g:A \rightarrow C^{2,\alpha} (S_f )$ by

$$g(v,\mu )= \lambda T((v+1)^q -qv -1) .$$

Then $g$ is a non-linear compact map.  Then define   $H:A \rightarrow C^{2,\alpha} (S_f )$ by

$$H(v,\mu )= v-\mu T(v) -g(v,\mu ).$$

Then  $v$ is a solution of $Eq_{\lambda}$ if and only if $H(v,\mu )=0$ (simply
apply $-\Delta +Id$ to this equation). Of course
one always has the trivial solution $v=0$, for any $\mu$. Moreover, $H$ satisfies the 
conditions to apply Krasnoselski's Theorem:  $g(v, \mu ) =o( \| v \| )$ uniformly on
$| \mu | < \epsilon$ (for any constant $\epsilon$). Then one notes that $1/\alpha$ is an eigenvalue of
$T$ if and only if $1-\alpha$ is an eigenvalue of $\Delta |_{ S_f}$, with the same eigenspaces.
Therefore it follows from Lemma 3.4 and Krasnoselski's Theorem \cite[Theorem 3.3.1]{Nirenberg}
that the points  $ (0,\mu_i ) $ with
$\mu_i = 1+il(n+il-1)$ are bifurcation points (and the only ones).

\vspace{.3cm}

To prove Theorem 1.1 we apply Rabinowitz' Theorem \cite[Theorem 3.4.1]{Nirenberg}.

We will prove

(*) If $C$ is the closure of the non-trivial ($v\neq 0$) solutions of $H(v,\mu )=0$ and $C_i$ is the
connected component of $C$ containing $(0,\mu_i )$ then $C_i$ does not contain any other 
of the bifurcation points. 

\vspace{.2cm}

Assuming (*) it follows from Rabinowitz' Theorem that $C_i$ is not compact in $A$. In order
to finish the proof of Theorem 1.1 it is enough to show that for any $\mu >\mu_i $ there 
exists $v\neq 0$ such that  $(v,\mu ) \in C_i$. 
If this were not the case then $C_i$ would be
contained in a region of the form $A \cap \left( C^{2,\alpha} (S_f ) \times [n/(q-1),D] \right) $ 
(there are no 
non-trivial solutions
of $Eq_{\lambda}$ for $\lambda \leq n/(q-1) $, \cite[Theorem 6.1]{Veron} ). 
But for uniformly bounded $\mu$, for all solutions of $H(v,\mu )=0$, $v$ is uniformly 
bounded above and bounded below away from $-1$ by \cite[Lemma 2.4 (b)]{Jin}. This would
imply that $C_i$ is compact and the proof follows.

\vspace{.2cm}

To prove (*) first note that the bifurcation points are the only points in $C$ with $v=0$. Next 
for any $v\in S_f$ let $z(v)$ be the number of  zeros of $\varphi$ where $v=\varphi \circ f$
(it might be $\infty$). Then
for each positive integer $k$ consider the set $B_k = \{ (v,\mu ) \in S_f  :
H(v, \mu )=0 , z(v)=k \}$. Note that if $H(v, \mu ) =0$  then $u=v+1 = \varphi \circ f$ where
$\varphi$ solves the ordinary differential equation (where $\lambda = \lambda (\mu )$)

$$-\left( l^2 \varphi '' (t) (1-t^2) +  \varphi ' (t)  ((1/2) c l^2 -l(n+l-1) t) \right) + \lambda \varphi = \lambda \varphi^q .$$

\noindent
If $v\neq 0$ then $u\neq 1$ and $u$ must have a minimum which must be $<1$ and a maximum $>1$ and 
therefore it must take the value 1 a positive finite number of times with non-zero derivative.  This implies:

i) All non-trivial solutions of $H(v,\mu )=0$ belong to $B_k$ for some positive integer $k$.

ii) Each $B_k$ is open in $C$.

\noindent
Also by the local bifurcation theorem of Crandall and Rabinowitz \cite[Theorem 3.2.2]{Nirenberg} the
non-trival solutions near the bifurcation point $(0,\mu_k )$ can be parametrized by $s$, $0< |s| < \varepsilon$,
by $(sp_k + o(s^2) , \lambda (s))$, which implies by Lemma 3.4 that non-trivial solutions near $(0,\mu_k )$
belong to $B_k$. Then define $C^0_i = (C_i \cap B_i ) \cup \{(0,\mu_i )\} $. If we prove that $C^0_i = C_i$
we are done; and to prove this it is enough to show that $C^0_i$ is open and closed in $C_i$, but this 
follows easily from the previous lines. 

We have therefore completed the proof of Theorem 1.1.

\qed

\section{Proof of Theorem 1.4}

Let $(M,g)$ and $(N,h)$ be closed Riemannian manifolds of constant scalar curvature with $\dim M=n$ and $\dim N=k$.  Let us denote with 
${\bf s}=s_g+s_h$ the scalar curvature of $(M \times N,g + h )$.  Let $f:M\rightarrow \re$ be an isoparametric function such that
the isoparametric hypersurface $S$ is a regular level set of $f$. Let $V=Vol(N,h)$ and
consider the Yamabe functional on $(M \times N, g + h )$ and  restrict it to $S_f$:

$$Y (u) =  \frac{a_{n+k} \|\nabla u \|_2^2 + {\bf s}  \|u\|_2^2}{\|u\|_p^2} \ V^{p-2/p} =\frac{Q(u)}{\|u\|_p^2}  \ V^{p-2/p} ,$$

with $2<p=p_{n+k}<p_{n}$. If $u,v \in S_f$ then

$$\frac{d Y(u+tv)}{dt} (0) = \frac{2 V^{p-2/p} }{\|u\|^2_p}\int\Big( - a_{n+k}\Delta_g u+{\bf s}  u-\|u\|_p^{-p}u^{p-1}Q(u)\Big)\ v \ dvol_g.$$

It follows from Lemma 3.1 that  $u\in S_f$ is a critical point of $Y|_{S_f}$ if and only if $u$ if a solution of the  
Yamabe equation $-a_{n+k}\Delta_gu +{\bf s}  u=\beta u^{p-1}$ with $\beta=\|u\|_p^{-p+2}Y(u)$. From now on we write $a_{n+k}$ as $a$.

\begin{Lemma}\label{existsuthatminimize} There is a smooth positive function $u\in S_f$ which minimizes
the Yamabe  functional restricted to $S_f$.
\end{Lemma}

\begin{proof}  Let $u_i \in H^2_1 (M ) \cap S_f$ be  a sequence of smooth non-negative functions  such that
$$\lim_{i\to\infty}Y(u_i)={\inf}_{v\in S_f}Y(v)=\alpha.$$

Assume $\|u_i\|_p=1$.   Since $p>2$ and $p<p_n$ one sees that
 that the sequence $u_i$ is a bounded sequence in $H^2_1$ and apply the Rellich-Kondrakov theorem to see
 that there is a  subsequence of $u_i$  that converges in $L_p$ to a function $u$. Then necessarily $u$ is a
 smooth positive solution of the Yamabe equation. This was actually carried out in  \cite[Proposition 2.2]{Akutagawa}. 
 But $u\in S_f$ because $H^2_1 (M) \cap S_f$ is a closed subset of $H^2_1 (M )$.
 
\end{proof}

\begin{Lemma}\label{1notminimize}  Assume that there is a non-constant  eigenfunction of the Laplacian
$w\in S_f$ with $\Delta w = \mu  w$. If ${\bf s}  > \frac{-a \mu}{p-2} $ then the constant
function is not the minimizer of the Yamabe functional  restricted to  $S_f$.

\end{Lemma}

\begin{proof}  Let $y(t)= Y(1+tw)$. Then $y'(0)=0$ (this is essentially the fact that $g+h$ has constant scalar curvature).
And

 $$\frac{d^2 y}{dt^2} (0) =
2(Vol (M\times N ,g+h))^{-\frac{2}{p}}\Big[(2-p){\bf s} -\mu.a\Big]\int w^2 dvol_g.$$

It follows from the hypothesis that the expression  above is negative, which  proves the Lemma.

\end{proof}

\begin{proof}[Proof of Theorem 1.4] We are considering 
 an isoparametric function  $f:M\longrightarrow \re$ such that $S$ is one of its regular level sets
 and  $\lambda_0 <0$ is the first negative eigenvalue of $\Delta_g|_{S_f}$ (see Proposition 3.2).  Then the previous Lemma 
 tells us that (under the conditions of Theorem 1.4) the constant functions do not minimize the Yamabe functional  restricted to  $S_f$.
 Hence Lemma \ref{existsuthatminimize} says that there exists  a positive non-constant  function $u\in S_f$ which  minimizes $Y$
 in $S_f$. 
 Therefore $u$, seen as a function from $M\times N$ to $\re_{>0}$,  is a solution of the Yamabe equation in 
 $(M\times N,g+h)$ and $S$ is included in a level set of $u$.

\end{proof}

\end{document}